\documentclass[oneside,11pt,reqno]{amsart}
\usepackage{amssymb,amsmath,amsthm,bbm,enumerate,mdwlist,url,multirow,hyperref}
\usepackage[pdftex]{graphicx}
\usepackage[shortlabels]{enumitem}

\theoremstyle{definition}
\newtheorem{definition}{Definition}
\theoremstyle{theorem}
\newtheorem{proposition}[definition]{Proposition}
\newtheorem{lemma}[definition]{Lemma}
\newtheorem{theorem}[definition]{Theorem}

\numberwithin{equation}{section}
\numberwithin{definition}{section}
\theoremstyle{remark}
\newtheorem{remark}[definition]{Remark}
\def\PP{\mathsf P}
\def\QQ{\mathsf Q}
\def\AA{\mathcal A}
\def\EE{\mathsf E}

\def\supp{\mathrm{supp}}
\def\Zh{\mathbb{Z}_h}
\def\Exp{\mathrm{Exp}}

\def\measure{\lambda}
\def\diffusion{\sigma^2}
\begin{document}
\title{Non-random overshoots of L\'evy processes}

\author{Matija Vidmar}
\address{Department of Statistics, University of Warwick, UK}
\email{m.vidmar@warwick.ac.uk}

\thanks{The support of the Slovene Human Resources Development and Scholarship Fund under contract number 11010-543/2011 is recognized. The author thanks Saul Jacka for his numerous comments and suggestions which have helped improve the presentation of this article. An extended abstract of this work is forthcoming in the 2013 proceedings of the Actuarial and Financial Mathematics Conference.}

\begin{abstract}
The class of L\'evy processes for which overshoots are almost surely constant quantities is precisely characterized.
\end{abstract}



\maketitle

\section{Introduction}\label{section:Introduction}
Fluctuation theory represents one of the most important areas within the study of L\'evy processes, with applications in finance, insurance, dam theory etc. \cite{kyprianou} A key result, then, is the Wiener-Hopf factorization, particularly explicit in the spectrally negative case, when there are no positive jumps, a.s. \cite[Section 9.46]{sato} \cite[Chapter VII]{bertoin}.

What makes the analysis so much easier in the latter instance, is the fact that the overshoots $(R_x)_{x\geq 0}$ \cite[p. 369]{sato} over a given level are known \emph{a priori} to be constant and equal to zero. As we shall see, this is also the only class of L\'evy processes for which this is true (see Lemma~\ref{lemma:supremumcontinuity}). But it is not so much the exact values of the overshoots that matter, as does the fact that these values are non-random (and known). It is therefore natural to ask if there are any other L\'evy processes having constant overshoots (a.s.) and, moreover, what \emph{precisely} is the class having this property. 

Of course, in the existing literature one finds expressions regarding the distribution of the overshoots. For example, \cite[p. 369, Theorem 49.1]{sato} gives the double Laplace transform $u\int_{(0,\infty)} e^{-ux}\EE[e^{-qR_x}]dx$ ($\{u,q\}\subset (0,\infty)$) in terms of the Wiener-Hopf factors. Similarly, in \cite{doneykyprianou} we find an expression for the law of the overshoot in terms of the L\'evy measure, but only after it has been integrated against the bivariate renewal functions. Unfortunately, neither of these seem immediately useful in answering the question posed above. 

Further to this, the asymptotic study of quantities at first passage above a given level has been undertaken in \cite{doneykyprianou,kyprianoupardorivero} and behaviour just prior to first passage has also been investigated, see, e.g. \cite[p. 378, Remark 49.9]{sato} and \cite[Chapter 7]{kyprianou}. On the other hand it appears that the (natural) question, outlined above, has not yet received due attention. 

The answer to it, presented in this paper, is as follows: for the overshoots of a L\'evy process to be almost surely constant (conditionally on the process going above the level in question), it is both necessary and sufficient that \emph{either} the process has no positive jumps (a.s.) \emph{or} for some $h>0$, it is compound Poisson, living on the lattice $\Zh:=h\mathbb{Z}=\{hk:k\in \mathbb{Z}\}$, and can only jump up by $h$. 

A more exhaustive statement of this result, which derives the same conclusion from substantially weakened hypotheses, is contained in Theorem~\ref{theorem:non-random-overshoots} of Section~\ref{section:Notation-and-statement-of-result}, which also introduces the required notation. Section~\ref{section:Proof-of-theorem} supplies the proof. Finally, Appendix~\ref{appendix:conditioning_lemmas} contains a result concerning conditional expectation, Proposition~\ref{proposition:lemma_on_conditioning}, which is used in the proof, but is also interesting in its own right. 

\section{Notation and statement of result}\label{section:Notation-and-statement-of-result}
Throughout we work on a filtered probability space $(\Omega,\mathcal{F},\mathbb{F}=(\mathcal{F}_t)_{t\geq 0},\PP)$, which satisfies the standard assumptions (i.e. the $\sigma$-field $\mathcal{F}$ is $\PP$-complete, the filtration $\mathbb{F}$ is right-continuous and $\mathcal{F}_0$ contains all $\PP$-null sets). We let $X=(X_t)_{t\geq 0}$ be a L\'evy process on this space with characteristic triplet $(\sigma^2,\lambda,\mu)_{\tilde{c}}$ relative to some cut-off function $\tilde{c}$ \cite[p. 2]{kyprianou} \cite[pp. 37-39]{sato}. This means that $X$ is an $\mathbb{F}$-adapted process, with stationary, independent increments relative to $\mathbb{F}$, $\PP(X_0=0)=1$ and $X$ is c\`adl\`ag (i.e. right-continuous and admitting left limits) $\PP$-a.s. Then $\diffusion$ is the diffusion coefficient, $\measure$ is the L\'evy measure and $\mu$ is the drift (relative to $\tilde{c}$).  Note that, by definition, $X$ is compound Poisson, if $\diffusion=0$, $\measure(\mathbb{R})\in (0,\infty)$ and (with $\tilde{c}=0$) $\mu=0$.  The supremum process of $X$ is denoted  $\overline{X}_t:=\sup\{X_s:s\in [0,t]\}$ ($t\geq 0$).

Next, for $x\in\mathbb{R}$ introduce $T_x:=\inf \{t\geq 0:X_t\geq x\}$ (respectively $\hat{T}_x:=\inf \{t\geq 0:X_t>x\}$), the first entrance time of $X$ to $[x,\infty)$ (respectively $(x,\infty)$). We will  informally refer to $T_x$ and $\hat{T}_x$ as the \emph{times of first passage} above the level $x$. 

$\mathcal{B}(S)$ will always denote the Borel $\sigma$-field of a topological space $S$; $\supp(m)$ the support of a measure $m$ thereon \cite[p. 9]{kallenberg}; we shall say $m$ is carried by $A\in \mathcal{B}(S)$, if $m(S\backslash A)=0$; $\delta_x:=(A\mapsto \mathbbm{1}_A(x))$, mapping $\mathcal{B}(S)$ into $[0,1]$, is the Dirac measure at $x\in S$. For a random element $R:(\Omega,\mathcal{F})\to (S,\mathcal{S})$, $\PP_R$ is the image measure $\PP\circ R^{-1}$ \cite[p. 24]{kallenberg}. $\mathcal{S}^\star$ denotes the universal completion of a $\sigma$-field $\mathcal{S}$. If $\mu$ is furthermore a measure on $\mathcal{S}$, then the completion of the $\sigma$-field $\mathcal{S}$ relative to the measure $\mu$, is denoted $\overline{\mathcal{S}}^\mu$, while $\overline{\mu}$ is the unique extension of $\mu$ to $\overline{\mathcal{S}}^\mu$ \cite[p. 13]{kallenberg}.

The next definition introduces the continuous-time analogue (modulo a spatial scaling) of a right-continuous integer-valued random walk (for which see, e.g., \cite{brown}):

\begin{definition}[Upwards skip-free L\'evy chain]\label{definition:USF}
$X$ is said to be an \emph{upwards skip-free L\'evy chain}, if it is a compound Poisson process, and for some $h>0$, $\supp(\lambda)\subset \Zh$ and $\supp(\lambda\vert_{\mathcal{B}((0,\infty))})=\{h\}$. 
\end{definition}
Finally, the following notion, which is a rephrasing of ``being almost surely constant conditionally on a given event'', will prove useful:

\begin{definition}[$\PP$-triviality]\label{definition:Ptrivial}
Let $S\ne\emptyset$ be any measurable space, whose $\sigma$-algebra $\mathcal{S}$ contains the singletons. An $S$-valued random element $R$ is said to be $\PP$-\emph{trivial on} an event $A\in\mathcal{F}$ if there exists $r\in S$ such that $R=r$ $\PP$-a.s. on $A$ (i.e. $\PP(\{R=r\}\cap A)=\PP(A)$; equivalently, the push-forward measure $(B\mapsto \PP(A\cap R^{-1}(B)))$, defined on $\mathcal{S}$, is carried by $\{r\}$, not excluding the case when $\PP(A)=0$). The random element $R$ may only be defined on some $B\supset A$ (in which case $R$ should be measurable with respect to the trace $\sigma$-algebra $\{B\cap G:G\in \mathcal{F}\}$). 
\end{definition}
Thanks to Definitions~\ref{definition:USF} and~\ref{definition:Ptrivial}, we can now state succinctly the main result of this paper:

\begin{theorem}[Non-random position at first passage time]\label{theorem:non-random-overshoots}
The following are equivalent:
\begin{enumerate}[(a)]
\item\label{non-random:a} For some $x>0$, $X(T_x)$ is $\PP$-trivial on $\{T_x<\infty\}$.
\item\label{non-random:b} For all $x\in\mathbb{R}$, $X(T_x)$ is $\PP$-trivial on $\{T_x<\infty\}$.
\item\label{non-random:c} For some $x\geq 0$, $X(\hat{T}_x)$ is $\PP$-trivial on $\{\hat{T}_x<\infty\}$ and $\PP$-a.s. strictly positive thereon.
\item\label{non-random:d} For all $x\in\mathbb{R}$, $X(\hat{T}_x)$ is $\PP$-trivial on $\{\hat{T}_x<\infty\}$.
\item\label{non-random:e} Either $\measure((0,\infty))=0$ or $X$ is an upwards skip-free L\'evy chain. 
\end{enumerate}
If so, then the exceptional sets in \ref{non-random:b} and \ref{non-random:d} can actually be chosen \emph{not to depend on $x$}; i.e.  outside a $\PP$-negligible set, \emph{for each $x\in\mathbb{R}$}, $X(T_x)$ (respectively $X(\hat{T}_x)$) is constant on $\{T_x<\infty\}$ (respectively $\{\hat{T}_x<\infty\}$).
\end{theorem}

\begin{remark}
In \ref{non-random:c}, if $x>0$, then $X(\hat{T}_x)$ is automatically $\PP$-a.s. strictly positive on $\{\hat{T}_x<\infty\}$.
\end{remark}
Finally, we make the following general notation explicit: $\mathbb{N}:=\{1,2,3,\ldots\}$, $\mathbb{N}_0=\mathbb{N}\cup \{0\}$, $\mathbb{R}^+:=(0,\infty)$, $\mathbb{R}_+:=[0,\infty)$, $\mathbb{R}^-:=(-\infty,0)$ and  $\mathbb{R}_-:=(-\infty,0]$; while for $q\in (0,\infty)$, $\Exp(q)$ denotes the exponential law (mean $1/q$);  the symbol $\perp$ is used to indicate stochastic independence (relative to the probability measure $\PP$, or some conditional measure $\PP(\cdot\vert A)$ (with $A\in \mathcal{F}$ and $\PP(A)>0$) derived therefrom, depending on the context); $B(x,\delta)$ is the open ball, centre $x\in \mathbb{R}$, radius $\delta>0$; and $\lceil x\rceil:=\inf\{k\in\mathbb{Z}: k\geq x\}$ ($x\in \mathbb{R}$) is the ceiling function. We will say positive for strictly positive, exceeding for strictly exceeding, decreasing for strictly decreasing and so on. 

Furthermore, it will at times be convenient to work with the canonical space $\mathbb{D}:=\{\omega\in \mathbb{R}^{[0,\infty)}:\omega\text{ is c\`adl\`ag}\}$ of  c\`adl\`ag paths, mapping $[0,\infty)$ into $\mathbb{R}$. Then $\mathcal{H}$ will denote the $\sigma$-field generated by all the evaluation maps, whereas for $\omega\in\mathbb{D}$, $\overline{\omega}$ will be the supremum process of $\omega$ (i.e. $\overline{\omega}(t):=\sup\{\omega(s):s\in [0,t]\}$, $t\geq 0$), and further for $a\in \mathbb{R}$, $T_a(\omega):=\inf \{t\geq 0:\omega(t)\geq x\}$ will be the first entrance time of $\omega$ into the set $[a,\infty)$. Context shall make it clear when $T_a$ will be seen as the latter mapping,  $T_a:\mathbb{D}\to [0,\infty]$, and when as the first entrance time of $X$ into $[a,\infty)$, as per above.

\section{Proof of theorem}\label{section:Proof-of-theorem}

\begin{remark}
$T_x$ and $\hat{T}_x$ are $\mathbb{F}$-stopping times for each $x\in\mathbb{R}$ (apply the D\'ebut Theorem \cite[p. 101, Theorem 6.7]{kallenberg}) and $\PP(T_x=0\text{ for all }x\in\mathbb{R}_-)=1$. Moreover, $\PP(T_x<\infty\text{ for all }x\in\mathbb{R})=1$, whenever $X$ either drifts to $+\infty$ or oscillates. If not, then $X$ drifts to $-\infty$ \cite[p. 255, Proposition 37.10]{sato} and on the event $\{T_x=\infty\}$ one has $\lim_{t\to T_x}X(t)=-\infty$ for each $x\in\mathbb{R}$, $\PP$-a.s. 
\end{remark}
For the most part we find it more convenient to deal with the collection $(T_x)_{x\in\mathbb{R}}$, rather than $(\hat{T}_x)_{x\in\mathbb{R}}$, even though this makes certain measurability issues more involved. 

\begin{remark}
Note that whenever $0$ is regular for $(0,\infty)$ (i.e. $\PP(\hat{T}_0=0)=1$), then for each $x\in\mathbb{R}$, $T_x=\hat{T}_x$ $\PP$-a.s. (apply the strong Markov property \cite[p. 278, Theorem 40.10]{sato} at the time $T_x$). For conditions equivalent to this, see \cite[p. 142, Theorem 6.5]{kyprianou}. Conversely, if $0$ is irregular for $(0,\infty)$, then by Blumenthal's $0-1$ law \cite[p. 275, Proposition 40.4]{sato}, $\PP$-a.s., $\hat{T}_0>0=T_0$.
\end{remark}
We now give two lemmas. The second concerns continuity of the supremum process $\overline{X}$. Since its formulation requires the relevant subsets of the sample space to be measurable, the first lemma establishes this.

In the next lemma, for a process $Y=(Y_t)_{t\geq 0}$, we agree $Y_{0-}:=Y_0$ and $Y_{t-}=\lim_{s\uparrow t}Y_s$ ($t>0$), whenever these limits exist.

\begin{lemma}
Let $(\Omega',\mathcal{G},\mathbb{G}=(\mathcal{G}_t)_{t\geq 0},\QQ)$ be a (respectively complete, i.e. $\QQ$ is complete and $\mathcal{G}_0$ contains all $\QQ$-null sets) filtered probability space. Suppose $Y$ is a $\mathbb{G}$-adapted and (respectively $\QQ$-a.s.) c\`adl\`ag  process. Then (with $\Omega_0$ being the (respectively $\QQ$-almost sure) event on which $Y$ is c\`adl\`ag), for each $\epsilon>0$ and $t\geq 0$, $A_\epsilon:=\cup_{s\in [0,t]}\{Y_s-Y_{s-}>\epsilon\}\cap \Omega_0\in \mathcal{G}_t$. As a consequence of this, the sets $\{Y\text{ is continuous}\}=\{Y_{t-}=Y_{t}\text{ for all }t\geq 0\}$ and $\{Y\text{ has no positive jumps}\}=\{Y_{t-}\geq Y_t\text{ for all }t\geq 0\}$ belong to $\mathcal{G}$. 
\end{lemma}
\begin{proof}
Define in addition $B_\epsilon:=\cup_{s\in [0,t]}\{Y_s-Y_{s-}\geq \epsilon\}\cap \Omega_0$ ($\epsilon>0$). Then, on the one hand, by the c\`adl\`ag property (respectively outside a $\QQ$-negligible set): 
\begin{equation}\label{eq:auxiliary:jump:measurability:one}
A_\epsilon\subset \cup_{n\in\mathbb{N}}F_n,\text{ where }F_n:=\cap_{N\in\mathbb{N}}\cup_{\{s,r\}\subset (\mathbb{Q}\cap [0,t])\cup \{t\},s<r,r-s<1/N}\{Y_r-Y_s>\epsilon+1/n\}.
\end{equation}
On the other hand, again by the c\`adl\`ag property, for each $n\in\mathbb{N}$ (respectively outside a $\QQ$-negligible set): 
\begin{equation}\label{eq:auxiliary:jump:measurability:two-dash}
F_n\subset B_{\epsilon+1/n}.
\end{equation}
Indeed, if $\omega\in F_n$ (and, respectively, $Y(\omega)$ is c\`adl\`ag), then for each $N\in \mathbb{N}$ we may choose a pair of real numbers $(s_N,r_N)$, $0\leq s_N<r_N\leq t$, $r_N-s_N<1/N$, with $Y_{r_N}-Y_{s_N}>\epsilon+1/n$. Since $[0,t]$ is compact, there is some accumulation point $s^\star$ for the sequence $(s_N)_{N\geq 1}$, and, by passing to a subsequence, we may assume without loss of generality that $s_N\to s^\star$ as $N\to\infty$. Moreover, by right-continuity, it is necessary that there is some natural $M$, with $s_N<s^\star$ for all $N\geq M$; whereas  by the existence of left-hand limits, it will also be necessary that there is some natural $M$, with $s^\star<r_N$ for all $N\geq M$. Then, by passing to the limit, it follows that $Y_{s^\star}(\omega)-Y_{s^\star-}(\omega)\geq \epsilon+1/n$. From \eqref{eq:auxiliary:jump:measurability:two-dash}, we conclude that (respectively $\QQ$-a.s.):
\begin{equation}\label{eq:auxiliary:jump:measurability:two}
\cup_{n\in\mathbb{N}}F_n\subset \cup_{n\in\mathbb{N}}B_{\epsilon+1/n}=A_\epsilon.
\end{equation}
Combining \eqref{eq:auxiliary:jump:measurability:one} and \eqref{eq:auxiliary:jump:measurability:two} we obtain (respectively by completeness) $A_\epsilon\in \mathcal{G}_t$. 

The final assertion of the lemma follows at once. 
\end{proof}

\begin{lemma}[Continuity of the running supremum]\label{lemma:supremumcontinuity}
The supremum process $\overline{X}$ is continuous ($\PP$-a.s.), if and only if $X$ has no positive jumps ($\PP$-a.s). In particular, if $X(T_x)=x$ $\PP$-a.s. on $\{T_x<\infty\}$ for each $x>0$, then $\overline{X}$ is continuous and hence $X$ has no positive jumps, $\PP$-a.s. 
\end{lemma}
\begin{proof}
We first show the validity of the equivalence. Sufficiency of the ``no positive jumps'' condition is immediate. We prove necessity by contradiction: suppose then, that $X$ had positive jumps with a positive probability and its supremum process was $\PP$-a.s. continuous. Then, for some $a>0$, $X$ would have a jump exceeding $a$ with a positive probability and necessarily we would have $\lambda((a,\infty))>0$. Moreover, by the L\'evy-It\^o decomposition, one may write, $\PP$-a.s., $X=X^1+X^2$ as an independent sum, where $X^2$ is a compound Poisson process of the positive jumps of $X$ exceeding (i.e. of height $>$) $a$ and $X^1=X-X^2$ is whatever remains (see e.g. \cite[p. 108, Theorem 2.4.16]{applebaum} and the results leading thereto, in particular \cite[p. 99, Theorem 2.4.6]{applebaum}). 

Next, let $S$ be the supremum process of $\vert X^1\vert$ and $T$ be the first jump time of $X^2$. By right-continuity of the sample paths, for some $t>0$, $\PP(\{S_t<a/2\})>0$. Further, by independence, and the fact that $T\sim \mathrm{Exp}(\lambda((a,\infty)))$ \cite[p. 101, Theorem 2.3.5(1)]{applebaum}, one has $\PP(\{S_t<a/2\}\cap \{T<t\})>0$. Hence, with a positive probability, $X$ will attain a new supremum (on $[0,t]$) by a jump in $\overline{X}$, which is a contradiction. 

Finally, suppose $X(T_x)=x$ $\PP$-a.s. on $\{T_x<\infty\}$ for each $x>0$. Then the supremum process $\overline{X}$ is a.s. continuous. Indeed, suppose not. Then with a positive probability $\overline{X}$ would have a jump, and therefore, for some pair of rationals $r_1,r_2$ with $0<r_1<r_2$, there would be a jump of $\overline{X}$ over $(r_1,r_2)$ with a positive probability. Then, on this event $X(T_{(r_1+r_2)/2})\geq r_2>(r_1+r_2)/2$, a contradiction. 
\end{proof}
Having established this lemma, the first main step towards the proof of Theorem~\ref{theorem:non-random-overshoots} is the following:

\begin{proposition}[$\PP$-triviality of $X(T_x)$]\label{proposition:Ptriviality}
The random variable $X(T_x)$ (defined on $\{T_x<\infty\}$) is $\PP$-trivial on $\{T_x<\infty\}$ for each $x>0$, if and only if:
\newline \noindent either
\begin{enumerate}[(a)]
\item\label{Ptrivial:a} $X$ has no positive jumps ($\PP$-a.s.) (equivalently: $\lambda((0,\infty))=0$) 
\end{enumerate}
or
\begin{enumerate}[resume*]
\item\label{Ptrivial:b} $X$ is compound Poisson and for some $h>0$, $\supp(\lambda)\subset \Zh$ and $\supp(\lambda\vert_{\mathcal{B}((0,\infty))})=\{h\}$
\end{enumerate}
(conditions \ref{Ptrivial:a} and  \ref{Ptrivial:b} being mutually exclusive). If so, then $X(T_x)=x$ on $\{T_x<\infty\}$ for each $x\geq 0$  ($P$-a.s.) under \ref{Ptrivial:a} and $X(T_x)=h\lceil x/h\rceil$ on $\{T_x<\infty\}$ for each $x\geq 0$  ($\PP$-a.s.) under \ref{Ptrivial:b}.
\end{proposition}
\begin{remark}
Note that, under \ref{Ptrivial:b}, $\PP(\{X_t\in \Zh\text{ for all }t\geq 0\})=1$. This follows by \cite[p. 149, Corollary 24.6]{sato} and sample path right-continuity. 
\end{remark}
The main idea behind the proof of Proposition~\ref{proposition:Ptriviality} is to appeal first to Lemma~\ref{lemma:supremumcontinuity} for the case when, for all $x>0$, $X(T_x)=x$ $\PP$-a.s. on $\{T_x<\infty\}$. This gives \ref{Ptrivial:a}.  Then we treat separately the compound Poisson case; in all other instances the L\'evy-It\^o decomposition and the well-established path properties of L\'evy processes yield the claim. Intuitively, for a L\'evy process to cross over every level in a non-random fashion, either it does so necessarily continuously when there are no positive jumps (cf. also \cite[p. 274, Proposition 6.1.2]{kolokoltsov}), or, if there are, then it must be forced to live on the lattice $\Zh$ for some $h>0$ and only jump up by $h$. Formally:
\begin{proof}
Assume, without loss of generality, that $X$ is c\`adl\`ag with certainty (rather than just $\PP$-a.s.). Clearly conditions \ref{Ptrivial:a} and \ref{Ptrivial:b} are mutually exclusive, sufficiency of the conditions and the final remark of Proposition~\ref{proposition:Ptriviality} obtain by sample path right-continuity. With regard to the equivalence noted parenthetically in \ref{Ptrivial:a} see \cite[p. 346, Remark 46.1]{sato}. 

Necessity of the conditions from Proposition~\ref{proposition:Ptriviality} is shown as follows. Let $X(T_x)$ be $\PP$-trivial on $\{T_x<\infty\}$ for each $x>0$.

 Suppose first that for each $x>0$, $X(T_x)=x$ ($\PP$-a.s.) on $\{T_x<\infty\}$. Then by Lemma~\ref{lemma:supremumcontinuity}, \ref{Ptrivial:a} must hold. 

There remains the case when, for some $x>0$, $\PP(T_x<\infty)>0$ and there is a non-random $f(x)$ with $f(x)=X(T_x)>x$ $\PP$-a.s. on $\{T_x<\infty\}$. In particular, $X$ must have positive jumps, and for some $a>0$, $\beta:=\lambda((a,\infty))>0$. Use again the L\'evy-It\^o decomposition as in Lemma~\ref{lemma:supremumcontinuity} with $S$ denoting the supremum process of $\vert X^1\vert$ and $T$ the first jump time of $X^2$ (note that $T\sim \Exp(\beta)$). We will consider the following two cases separately:

\begin{enumerate}[leftmargin=2cm,label=(Case~\arabic{*}),ref=Case~\arabic{*}]
\item\label{non-random-overshoots:case1}  $X$ is \emph{not} compound Poisson, i.e. either $\lambda$ has infinite mass or $\sigma^2>0$, or if this fails (with $\tilde{c}=0$ as the cut-off function) $\mu\ne 0$. 
\item\label{non-random-overshoots:case2} $X$ is compound Poisson, i.e. the diffusion coefficient vanishes, $\sigma^2=0$, $\lambda$ is finite and  (with $\tilde{c}=0$ as the cut-off function) the drift $\mu=0$. 
\end{enumerate}
Consider first \ref{non-random-overshoots:case1}. By right-continuity of the sample paths, there is a $t>0$ with $\PP(\{S_t<a/4\})>0$. 

We next argue that, on the event: $$C:=\{T<t\}\cap \{S_t<a/4\},$$ which has positive probability, $X^1(T)$ is not $\PP$-trivial. We prove this by contradiction. More precisely, we shall find that assuming the converse, will contradict the following observation regarding the sample paths of $X^1$: the set of jump times of $X^1$ is dense, a.s., by \cite[p. 136, Theorem 21.3]{sato} when $\lambda$ has infinite mass; the sample paths of $X^1$ have locally infinite variation, a.s., by \cite[p. 140, Theorem 21.9(ii)]{sato} when $\sigma^2>0$; finally, $X^1$ has no non-degenerate intervals of constancy, a.s., when $\sigma^2=0$, $\lambda(\mathbb{R})<\infty$ but the drift is non-zero.

Indeed, suppose that $X^1(T)$ were to be $\PP$-trivial on the event $C$, so that there would be a (necessarily unique) $b\in (-a/4,a/4)$ such that $X^1(T)=b$ $\PP$-a.s. on $C$, i.e. $\PP(\{X^1(T)=b\}\cap C)=\PP(C)$. We next condition on $\mathcal{G}:=\sigma(T)$ by applying Proposition~\ref{proposition:lemma_on_conditioning_basic}. Specifically, we take, discarding, without loss of generality, the $\PP$-negligible set $\{T=\infty\}$, $Y:=T$ (so that $Y:(\Omega,\mathcal{F})\to (\mathbb{R}_+,\mathcal{B}(\mathbb{R}_+))$ and, of course, $\sigma(Y)\subset \mathcal{G}$) and $Z:=X^1$ (so that  $\sigma(Z)\perp \mathcal{G}$ and $Z:(\Omega,\mathcal{F})\to (\mathbb{D},\mathcal{H})$ --- recall from the end of Section~\ref{section:Notation-and-statement-of-result} notation pertaining to the space $(\mathbb{D},\mathcal{H})$). Finally, $f:\mathbb{R}_+\times\mathbb{D}\to\mathbb{R}$ is given by: $$f(s,\omega):=\mathbbm{1}_{\{b\}}(\omega(s))\mathbbm{1}_{[0,t)}(s)\mathbbm{1}_{[0,a/4)}(\max\{\overline{\omega}(t),-\overline{-\omega}(t)\}),\quad (s,\omega)\in \mathbb{R}_+\times\mathbb{D}.$$ Note that the latter is bounded and $\mathcal{B}(\mathbb{R})\otimes\mathcal{H}$/$\mathcal{B}(\mathbb{R})$-measurable by \cite[p. 5, 1.14 Remark]{karatzas} and since, owing to sample path right-continuity, $(\omega\mapsto \overline{\omega}(t))$ is $\mathcal{H}$/$\mathcal{B}(\mathbb{R}_+)$-measurable. Proposition~\ref{proposition:lemma_on_conditioning_basic} thus yields: $$\EE[f\circ (Y,Z)\vert \mathcal{G}]=g\circ Y,$$ where $g:=(y\mapsto \EE[f\circ (y,Z)])$, $g:\mathbb{R}_+\to \mathbb{R}$, is Borel measurable. Now, on the one hand: $$\EE[g\circ Y]=\int g d\PP_T=\int_0^\infty ds  \beta e^{-\beta s}\EE[f\circ (s,Z)]=\int_0^tds\beta e^{-\beta s}\PP(\{X^1(s)=b\}\cap \{S_t<a/4\}).$$ On the other hand: 
\begin{eqnarray*}
\EE[\EE[f\circ (Y,Z)\vert \mathcal{G}]]&=&\EE[f\circ (Y,Z)]=\PP(\{X^1(T)=b\}\cap C)=\PP(C)\\
&=&\PP(T<t)\PP(S_t<a/4)=\int_0^tds\beta e^{-\beta s}\PP(S_t<a/4).
\end{eqnarray*}
In summary, it follows that: $$\int_0^t ds\beta e^{-\beta s}\PP(\{X^1(s)=b\}\cap \{S_t<a/4\})=\int_0^t ds\beta e^{-\beta s}\PP(\{S_t<a/4\}).$$ Hence, Lebesgue-a.e. in $s\in (0,t)$, a.s. on $\{S_t<a/4\}$, $X^1(s)=b$. Now we can find for each rational $r\in (0,t)$ and $n\in \mathbb{N}$ an $x_n^r \in B(r,1/n)$ for which a.s. on $\{S_t<a/4\}$, $X^1(x_n^r)=b$. So a.s. on  $\{S_t<a/4\}$, on a dense countable subset of $(0,t)$, $X^1=b$. Thus by sample path right-continuity a.s. on  $\{S_t<a/4\}$, $X^1=b$ everywhere on $[0,t)$. Hence, on an event of positive probability, there are \emph{no} jump times on the whole of the interval $[0,t)$, the path has \emph{zero} variation over $[0,t)$ and is, moreover, \emph{constant} thereon, a contradiction.

We have thus established that $X^1(T)$ is \emph{not} $\PP$-trivial on the event $C$. 

Observe now that $X^2(T)$ is independent of $T$, both being jointly independent of $X^1$. Then $X^2(T)\perp \sigma(\mathbbm{1}_C,X^1(T))$, so that (for Borel subsets $A$ and $B$ of $\mathbb{R}$): 
\begin{equation*}
\PP(C\cap \{X^1(T)\in A\}\cap \{X^2(T)\in B\})\PP(C)=\PP(C\cap \{X^1(T)\in A\})\PP(C\cap \{X^2(T)\in B\}).
\end{equation*}
We conclude that the first jump of $X^2$, $X^2(T)$, is independent of $X^1(T)$, conditionally on $C$. The support of their sum $X(T)=X^1(T)+X^2(T)$ on $C$, is therefore the closure of the sum of their respective supports \cite[p. 148, Lemma 24.1]{sato} and as such contains at least two points. It follows that, on the stipulated event of positive probability, which is contained in $\{T_{a/2}<\infty\}$ and on which $T_{a/2}=T$, $X(T_{a/2})=X(T)$ is not $\PP$-trivial, a contradiction.

Consider now \ref{non-random-overshoots:case2}. Suppose furthermore that the support of $\lambda\vert_{\mathcal{B}((0,\infty))}$ were to contain at least two points $b<c$, say. Choose $\delta<b/2$ small enough such that $B(b,\delta)\cap B(c,\delta)=\emptyset$. The measure $\lambda$ must charge both these open balls, and hence the first jump can be in either one, each with a positive probability. Thus $X(T_{b/2}$) would not be $\PP$-trivial on the event $\{T_{b/2}<\infty\}$, a contradiction. Plainly, then, the support of $\lambda\vert_{\mathcal{B}((0,\infty))}$ is $\{h\}$ for some $h>0$. 

It only remains to show that $\lambda$ is supported by $\Zh$. To see this, suppose it were not. Then there would be an $x<0$ and a $\delta>0$, with $B(x,\delta)$ having a non-empty intersection with the support of $\lambda$ and an empty intersection with $\Zh$. With a positive probability $X$ would jump into $B(x,\delta)$ and then have a sequence of jumps of size $h$ upwards going above $h$ for the first time at a level distinct from $h$. With a positive probability, $X$ also goes above $h$ by making its first jump to $h$, a contradiction.

The proof is complete.
\end{proof}

The second (and last) main step towards the proof of Theorem~\ref{theorem:non-random-overshoots} consists in taking advantage of the temporal and spatial homogeneity of L\'evy processes. Thus the condition in Proposition~\ref{proposition:Ptriviality} is relaxed to one in which the $\PP$-triviality of the position at first passage is required for one $x>0$, rather than all. To shorten notation let us introduce:

\begin{definition}
For $x\in\mathbb{R}$, let $\QQ^x:\mathcal{B}(\mathbb{R})\to [0,\PP(T_x<\infty)]$, $$\QQ^x(B):=\PP(\{X(T_x)\in B\} \cap \{T_x<\infty\}), \quad B\in \mathcal{B}(\mathbb{R}),$$ be the (possibly subprobability) law of $X(T_x)$ on $\{T_x<\infty\}$ under $\PP$ on the space $(\mathbb{R},\mathcal{B}(\mathbb{R}))$. We also introduce the set: $$\AA:=\{x\in\mathbb{R}:\QQ^x\text{, which may have zero mass,  is carried by a singleton}\}.$$
\end{definition}

\begin{remark}\label{remark:theAAset}
Clearly $(-\infty,0]\subset \AA$ and for each $a\in \AA$, there exists an (unique, if $\PP(T_a<\infty)>0$) $f(a)$ such that: $$\QQ^a=\PP(T_a<\infty)\delta_{f(a)}.$$
\end{remark}
With this at our disposal, we can formulate our claim as:

\begin{proposition}\label{proposition:from_one_to_all}
Suppose $\AA\cap \mathbb{R}^+\ne \emptyset$. Then $\AA=\mathbb{R}$.
\end{proposition}
The proof of Proposition~\ref{proposition:from_one_to_all} will proceed in several steps, but the essence of it consists in establishing the intuitively appealing identity $\QQ^b(A)=\int d\overline{\QQ^c}(x_c)\QQ^{b-x_c}(A-x_c)$ for Borel sets $A$ and $c\in (0,b)$, see Lemma~\ref{lemma:an_integral_connection} below. This identity puts a constraint on the family of measures $(\QQ^a)_{a\in \mathbb{R}}$. In particular, it allows to demonstrate that $\AA$ is dense in the reals. Then we can appeal to quasi-left-continuity to conclude the proof. The main argument is thus fairly short, and a substantial amount of time is spent on measurability issues.

\begin{lemma}\label{lemma:an_integral_connection}
Let $b\in \mathbb{R}^+$, $c\in (0,b)$ and $A\in \mathcal{B}(\mathbb{R})$. Then: 
\begin{equation}\label{lemma:overshoot_measures_constraint}
\QQ^b(A)=\int d\overline{\QQ^c}(x_c)\QQ^{b-x_c}(A-x_c).
\end{equation}
\end{lemma}
\begin{proof}
If $\PP(T_c<\infty)=0$, then $\PP(T_b<\infty)=0$, $\QQ^b=\QQ^c=0$, and the claim is trivial. So assume, without loss of generality, that $\PP(T_c<\infty)>0$ and that $X$ is c\`adl\`ag with certainty (rather than just $\PP$-a.s.).

Let (on $\{T_c<\infty\}$): $\overset{\triangle}{X}:=(X(T_c+t)-X(T_c))_{t\geq 0}$ and $\overset{\triangle}{T}_y:=\inf\{t\geq 0:\overset{\triangle}{X}_t\geq y\}$ ($y\in\mathbb{R}$), while $\mathcal{F}_{T_c}':=\{B\cap \{T_c<\infty\}:B\in \mathcal{F}_{T_c}\}$ is  $\mathcal{F}_{T_c}$ lowered onto $\{T_c<\infty\}$. By the strong Markov property, $\overset{\triangle}{X}$ is independent of $\mathcal{F}_{T_c}'$ under $\PP(\cdot\vert \{T_c<\infty \})$. Then:
\begin{eqnarray*}
\QQ^b(A)&=&\EE[\mathbbm{1}_A\circ X(T_b)\mathbbm{1 }_{\{T_b<\infty\}}],\text{ by the definition of }\QQ^b,\\
&=&\EE[\mathbbm{1}_{\{\overset{\triangle}{X}(\overset{\triangle}{T}_{b-X(T_c)})+X(T_c)\in A\}}\mathbbm{1}_{\{\overset{\triangle}{T}_{b-X(T_c)}<\infty\}}\mathbbm{1}_{\{T_c<\infty\}}], \text{ since }T_b=T_c+\overset{\triangle}{T}_{b-X(T_c)},\\
&=&\PP(T_c<\infty)\times \EE^{\PP(\cdot\vert \{T_c<\infty\})}\left[\EE^{\PP(\cdot\vert \{T_c<\infty\})}\left [\mathbbm{1}_{\{\overset{\triangle}{X}(\overset{\triangle}{T}_{b-X(T_c)})+X(T_c)\in A\}}\mathbbm{1}_{\{\overset{\triangle}{T}_{b-X(T_c)}<\infty\}}\vert \mathcal{F}_{T_c}'\right]\right],\\
&&\text{ by the tower property and the definition of the conditional measure }\PP(\cdot\vert \{T_c<\infty\}),\\
&=&\int d\overline{\QQ^c}(x_c)\QQ^{b-x_c}(A-x_c),\\
&&\text{by the strong Markov property \& Proposition~\ref{proposition:lemma_on_conditioning} (see below).}
\end{eqnarray*}
We now specify precisely how the strong Markov property and Proposition~\ref{proposition:lemma_on_conditioning} are applied here, this not being completely trivial. Recall from the end of Section~\ref{section:Notation-and-statement-of-result} the notation pertaining to the space $(\mathbb{D},\mathcal{H})$.

The probability space we will be working on is $(\{T_c<\infty\},\mathcal{F}_{\{T_c<\infty\}},\PP(\cdot\vert \{T_c<\infty\}))$, where $\mathcal{F}_{\{T_c<\infty\}}:=\{B\cap \{T_c<\infty\}:B\in \mathcal{F}\}$, and it is complete, since $(\Omega,\mathcal{F},\PP)$ is. Further, define (on $\{T_c<\infty\}$) $Y:=X(T_c)$; $Z:=\overset{\triangle}{X}$ and $f:\mathbb{R}\times\mathbb{D}\to\mathbb{R}$ by: $$f(x,\omega)=\mathbbm{1}_A(x+\omega(T_{b-x}(\omega)))\mathbbm{1}_{[0,\infty)}(T_{b-x}(\omega)),\quad (x,\omega)\in \mathbb{R}\times\mathbb{D},$$ where we let $\omega(\infty)=\omega(0)$ for definiteness.\footnote{The reader is cautioned not to confuse the mapping $f$, which is introduced here solely for the purposes of establishing how Proposition~\ref{proposition:lemma_on_conditioning} is applied in obtaining \eqref{lemma:overshoot_measures_constraint}, with the notation from Remark~\ref{remark:theAAset}. Indeed, the context will always make it clear which $f$ we are referring to.} 

Now, the random element $Z:(\{T_c<\infty\},\mathcal{F}_{\{T_c<\infty\}})\to (\mathbb{D},\mathcal{H})$ is independent of $\mathcal{G}:=\mathcal{F}_{T_c}'$, whereas the random element  $Y:(\{T_c<\infty\},\mathcal{F}_{\{T_c<\infty\}})\to (\mathbb{R},\mathcal{B}(\mathbb{R}))$ is measurable with respect to $\mathcal{F}_{T_c}'$. Measurability of $Y$ is a consequence of \cite[p. 5, 1.13 Proposition \& p. 9, 2.18 Proposition]{karatzas} and the D\'ebut Theorem \cite[p. 101, Theorem 6.7]{kallenberg} and measurability of $Z$ follows similarly. 

We next show that $f$ is $(\mathcal{B}(\mathbb{R})\otimes \mathcal{H})^\star/\mathcal{B}(\mathbb{R})$-measurable. First note that:
\begin{enumerate}
\item  $(x,\omega)\mapsto \omega+x$ is $\mathcal{B}(\mathbb{R})\otimes \mathcal{H}$/$\mathcal{H}$-measurable (in fact continuous, compare \cite[p. 328, 1.17 Proposition \& p. 329, 1.23 Proposition]{jacod}), hence $(\mathcal{B}(\mathbb{R})\otimes \mathcal{H})^\star$/$\mathcal{H}^\star$-measurable, by \cite[(2) on p. 23]{meyer}.
\item By the D\'ebut Theorem, for every $b\in\mathbb{R}$, $T_b$ is a stopping time of the augmented (with respect to \emph{any} probability measure) right-continuous modification of the canonical filtration $\mathbb{H}=(\mathcal{H}_t)_{t\geq 0}$ on $\mathbb{D}$ /where $\mathcal{H}_t$ is generated by the evaluation maps up to, and including, time $t$, $t\geq 0$/. Hence $(\omega\mapsto T_b(\omega))$ is $\mathcal{H}^\star$/$\mathcal{B}([0,\infty])$-measurable. 
\end{enumerate}
It follows that $(x,\omega)\mapsto T_b(\omega+x)=T_{b-x}(\omega)$ is $(\mathcal{B}(\mathbb{R})\otimes \mathcal{H})^\star$/$\mathcal{B}([0,\infty])$-measurable (as a composition). Next:
\begin{enumerate}
\item $(x,\omega)\mapsto(\omega, \mathbbm{1}_{[0,\infty)}(T_{b-x}(\omega))T_{b-x}(\omega))$ is $(\mathcal{B}(\mathbb{R})\otimes \mathcal{H})^\star$/$\mathcal{H}\otimes \mathcal{B}(\mathbb{R}_+)$-measurable.
\item $(\omega,t)\mapsto \omega(t)$ is $\mathcal{H}\otimes \mathcal{B}(\mathbb{R}_+)$/$\mathcal{B}(\mathbb{R})$-measurable (indeed, if $X$ is the coordinate process on $\mathbb{D}$, then this is the mapping $(\omega,t)\mapsto X(\omega,t)$, which is measurable by \cite[p. 5, Proposition~1.13]{karatzas}). 
\end{enumerate}
Therefore $(x,\omega)\mapsto \omega(T_{b-x}(\omega))$ is $(\mathcal{B}(\mathbb{R})\otimes \mathcal{H})^\star$/$\mathcal{B}(\mathbb{R})$-measurable (as a composition, with the above convention for $\omega(\infty)$). The required measurability of $f$ now follows from measurability of addition and multiplication. 

We are now in a position to apply Proposition~\ref{proposition:lemma_on_conditioning}. We have: 
\begin{eqnarray*}
&&\PP(T_c<\infty)\EE^{\PP(\cdot\vert \{T_c<\infty\})}[\EE^{\PP(\cdot\vert \{T_c<\infty\})}[f\circ (Y,Z)\vert \mathcal{F}_{T_c}']]=\\
&=&\PP(T_c<\infty)\EE^{\PP(\cdot\vert \{T_c<\infty\})}[(y\mapsto \EE^{\PP(\cdot\vert \{T_c<\infty\})}[f\circ (y,Z)])\circ X(T_c)],\text{ by Proposition~\ref{proposition:lemma_on_conditioning}},\\
&=&\!\!\!\int\! d\overline{\QQ^c}(y)\EE^{\PP(\cdot\vert \{T_c<\infty\})}[f\circ (y,Z)],\text{ by the Image Measure Theorem  \cite[p. 121, Theorem 4.1.11]{dudley}},\\
&&\text{since }\overline{\QQ^c}\text{ coincides with the (subprobability) law of }X(T_c)\text{ on }(\mathbb{R},\overline{\mathcal{B}(\mathbb{R})}^{\QQ^c}).
\end{eqnarray*}
Note here that we need to work with the (subprobability) law of $X(T_c)$ on the space $(\mathbb{R},\overline{\mathcal{B}(\mathbb{R})}^{\QQ^c})$ /rather than $(\mathbb{R},\mathcal{B}(\mathbb{R}))$/, since we only know the integrand to be measurable with respect to $\overline{\mathcal{B}(\mathbb{R})}^{\QQ^c}$. 

Now, by the strong Markov property, $Z$ is also identical in law under the measure $\PP(\cdot\vert \{T_c<\infty\})$ to $X$ under the measure $\PP$ on the space $(\mathbb{D},\mathcal{H})$ and hence on the space $(\mathbb{D},\mathcal{H}^\star)$  /the extension of a law to the universal completion being unique \cite[(1) on p. 23]{meyer}/. Moreover, for any real $d$ and Borel set $D\subset \mathbb{R}$, the mapping $g_{d,D}:\mathbb{D}\to\mathbb{R}$ given by $(\omega\mapsto \mathbbm{1}_{D}(\omega(T_{d}(\omega)))\mathbbm{1}_{[0,\infty)}(T_{d}(\omega)))$ is $\mathcal{H}^\star$/$\mathcal{B}(\mathbb{R})$-measurable, by the same reasoning as above. Hence:
\begin{eqnarray*}
\EE^{\PP(\cdot\vert \{T_c<\infty\})}[f\circ (y,Z)]&=&\EE^{\PP(\cdot\vert \{T_c<\infty\})}[\mathbbm{1}_{A-y}\circ \overset{\triangle}{X}(\overset{\triangle}{T}_{b-y})\mathbbm{1}_{[0,\infty)}\circ \overset{\triangle}{T}_{b-y}]\\
&=&\EE^{\PP(\cdot\vert \{T_c<\infty\})}[g_{b-y,A-y}\circ Z]\\
&=&\EE^{\PP}[g_{b-y,A-y}\circ X]=\QQ^{b-y}(A-y),
\end{eqnarray*}
as required. 
\end{proof}
\vspace{0.5cm}
\noindent
\emph{Proof of Proposition~\ref{proposition:from_one_to_all}}.
Given $\AA\cap \mathbb{R}^+\ne \emptyset$, we wish to show the inclusion $\mathbb{R}^+\subset\mathcal{A}$. Assume, again without loss of generality, that $X$ is c\`adl\`ag with certainty (rather than just $\PP$-a.s.).

\begin{enumerate}[(i)]
\item First observe that $\PP(T_x=\infty)=1$ for some $x>0$, precisely when $\PP(T_x=\infty)=1$ for all $x>0$. This follows either by the strong Markov property of L\'evy processes \cite[p. 68, Theorem 3.1]{kyprianou} and mathematical induction or, alternatively, one can appeal directly to \cite[p. 155, Proposition 24.14(i)]{sato}. Therefore it is sufficient to consider the case when $\PP(T_{x}<\infty)>0$ for all $x\in\mathbb{R}$. 
\item\label{non-random:strong:proof:ii} Claim: 
\begin{enumerate}[label=($\natural$),ref=($\natural$)]
\item \label{(natural)} If $b\in\AA$, then for every $c\in (0,b)$: either $c\in \AA$ or $(0,b-c]\cap \AA\ne\emptyset$.
\end{enumerate}
To show this, let $b\in \AA$, $c\in (0,b)$ and take any $A\in\mathcal{B}(\mathbb{R})$. By Lemma~\ref{lemma:an_integral_connection}:
\begin{equation}\label{eq:non-random:strong:1}
\QQ^b(A)=\int d\overline{\QQ^c}(x_c)\QQ^{b-x_c}(A-x_c).
\end{equation}
On the other hand, since $b\in \AA$:
\begin{equation}\label{eq:non-random:strong:2}
\QQ^b(A)=\PP(T_b<\infty)\delta_{f(b)}(A).
\end{equation}
Combining \eqref{eq:non-random:strong:1} and \eqref{eq:non-random:strong:2}, we have $\int d\overline{\QQ^c}(x_c)\QQ^{b-x_c}(A-x_c)=\PP(T_b<\infty)\delta_{f(b)}(A)$, from which we conclude that $\overline{\QQ^c}$-a.e. in $x_c\in\mathbb{R}$, $Q^{b-x_c}$ assigns all its mass to $\{f(b)-x_c\}$. (Suppose not, then with $\overline{\QQ^c}$-positive measure in $x_c\in\mathbb{R}$, $\QQ^{b-x_c}(\mathbb{R}\backslash\{f(b)-x_c\})>0$, and hence $Q^b(\mathbb{R}\backslash \{f(b)\})>0$, a contradiction.) 

Next, if $b'\in \AA$ and $c'\in (0,b']$:
\begin{enumerate}[label=(*),ref=(*)]
\item \label{(*)} $\QQ^{c'}$ assigns all its mass to $[c',b')\cup \{f(b')\}$.
\end{enumerate}
Therefore $c\in \mathcal{A}$, or $\QQ^c$ cannot ascribe all its mass to $\{f(b)\}$ and hence $\QQ^c([c,b))>0$. In the latter case, for some $x_c\in  [c,b)$, $\QQ^{b-x_c}$ is carried by $\{f(b)-x_c\}$, whence $b-x_c\in \AA\cap (0,b-c]$. 
\item Let $x_0:=\inf \mathcal{A}\cap \mathbb{R}^+$. Then $x_0=0$. Indeed, if not, then \ref{(natural)} of \ref{non-random:strong:proof:ii}, applied to some $[x_0,\infty)\cap \mathcal{A}\ni b<3x_0/2$ and $c=3x_0/4$ (say), yields a contradiction. Therefore there exists a decreasing sequence $(x_n)_{n\in \mathbb{N}}$ in $\mathcal{A}\cap\mathbb{R}^+$ converging to $0$. 
\item Claim: $\mathcal{A}$ is dense in $\mathbb{R}$. If $f(x_n)\to 0$ as $n\to\infty$, this is obvious, since, 
\begin{enumerate}[label=(**),ref=(**)]
\item \label{(**)} with any $x\in \mathcal{A}$, $\cup_{n\in \mathbb{N}_0}[x+nf(x),(n+1)f(x)]\subset\mathcal{A}$, 
\end{enumerate}
by the strong Markov property and mathematical induction. Suppose the nonincreasing sequence $(f(x_n))_{n\in\mathbb{N}_0}$ does not converge to $0$. Then there is an $\epsilon>0$ and a natural $N$, such that $f(x_n)\geq \epsilon$ and $x_n<\epsilon$ for all $n\geq N$. In particular, by \ref{(*)}, $f(x_n)=f(x_N)$ for all $n\geq N$. Therefore $[x_n,f(x_N)]\subset\AA$ for all $n\geq N$ by \ref{(**)}.  Therefore $[0,f(x_N)]\subset \AA$ and upon exceeding any positive level less than or equal to $f(x_N)$ we land at $f(x_N)$ a.s. Hence, by the strong Markov property and mathematical induction, $\mathcal{A}=\mathbb{R}$.
\item So we may assume $\AA$ is dense. Now we use quasi-left-continuity of L\'evy processes \cite[p. 21, Proposition 7]{bertoin} as follows. Take any $x\in\mathbb{R}^+$ and a sequence $\AA\cap (0,x)\supset (x_n)_{n\geq 1}\uparrow x$. Introduce the $\mathbb{F}$-stopping time $S:=\inf \{t\geq 0: \overline{X}_t\geq x\}$. We then have $T_{x_n}\uparrow S$ (as $n\to\infty$). By quasi-left-continuity, it follows that $\lim_{n\to\infty}X(T_{x_n})=X(S)$ $\PP$-a.s. on $\{S<\infty\}$. Therefore, in fact, $S=T_x$ $\PP$-a.s. on $\{S<\infty\}$ (and hence on $\{T_x<\infty\}$), and, moreover $X(T_x)=\lim_{n\to\infty}f(x_n)$ $\PP$-a.s. on $\{T_x<\infty\}$. But this means, precisely, that $x\in\AA$. 
\end{enumerate}
The proof is complete.\qed

Finally we can combine the above into a proof of Theorem~\ref{theorem:non-random-overshoots}.
\vspace{0.5cm}
\newline
\emph{Proof of Theorem~\ref{theorem:non-random-overshoots}}. The statement is essentially contained in Propositions~\ref{proposition:Ptriviality} and~\ref{proposition:from_one_to_all}.  We only have to worry about \ref{non-random:c} and \ref{non-random:d}, since so far we have only considered the stopping times $T_x$. 

Now, \ref{non-random:c} implies for some $f(x)>0$, $X(\hat{T}_x)=f(x)$ $\PP$-a.s. on $\{\hat{T}_x<\infty\}$, therefore $X(T_{f(x)})=f(x)$ $\PP$-a.s. on $\{T_{f(x)}<\infty\}$ and hence \ref{non-random:a}. Conversely, \ref{non-random:e} implies \ref{non-random:d} by sample path right-continuity. \qed

\begin{remark}
Theorem~\ref{theorem:non-random-overshoots} characterizes the class of L\'evy processes for which overshoots are known \emph{a priori} and are non-random. Moreover, the original motivation for this investigation is validated by the fact that upwards skip-free L\'evy chains admit a fluctuation theory, which is just as explicit, almost (but not entirely) analogous to the spectrally negative case and which embeds (existing) results for right-continuous random walks into continuous time. These findings, however, are deferred to a forthcoming paper \cite{vidmar:fluctuation_theory}.
\end{remark}

\bibliographystyle{plain}
\bibliography{Biblio_non-random_overshoots}

\appendix
\section{Two lemmas on conditioning}\label{appendix:conditioning_lemmas}
Let $(\Omega,\mathcal{F},\PP)$ be a probability space. Recall that the symbol $\perp$ is used to indicate stochastic independence relative to the probability measure $\PP$, whereas the completion of a $\sigma$-field $\mathcal{S}$ relative to the measure $\mu$  is denoted $\overline{\mathcal{S}}^\mu$, $\overline{\mu}$ being the unique extension of $\mu$ to $\overline{\mathcal{S}}^\mu$.

\begin{proposition}[Basic lemma on conditioning]\label{proposition:lemma_on_conditioning_basic}
Let $Y:(\Omega,\mathcal{F}) \to (S,\mathcal{S})$ and $Z:(\Omega,\mathcal{F})\to (T,\mathcal{T})$ be two random elements, and $\mathcal{G}$ any sub-$\sigma$-algebra of $\mathcal{F}$, such that $\sigma(Y)\subset \mathcal{G}$ and $\sigma(Z)\perp \mathcal{G}$. Let $f$ be any bounded (or nonnegative, or nonpositive) $\mathcal{S}\otimes \mathcal{T}$/$\mathcal{B}([-\infty,+\infty])$-measurable mapping. Then for any $y\in S$, $f \circ (y,Z)$ is $\mathcal{F}$/$\mathcal{B}([-\infty,+\infty])$-measurable, $g:=(y\mapsto \EE[f\circ (y,Z)])$ is $\mathcal{S}$/$\mathcal{B}([-\infty,+\infty])$-measurable and, $\PP$-a.s., 

\begin{equation}\label{eq:loc_basic}
\EE[f\circ (Y,Z)\vert \mathcal{G}]=g\circ Y.
\end{equation}
\end{proposition}
\begin{proof}
Linearity and monotonicity of conditional expectation \cite[p. 143]{cinlar} show that the class of functions $f$ for which the conclusion of the lemma holds true is a monotone class. By the Functional Monotone Class Theorem \cite[p. 10, Theorem~2.19]{cinlar}, it is then sufficient to check its validity for $f=\mathbbm{1}_\Lambda$ with $\Lambda$ belonging to the $\pi$-system $\{A\times B:(A,B)\in \mathcal{S}\times \mathcal{T}\}$ generating $\mathcal{S}\otimes \mathcal{T}$. In that case \eqref{eq:loc_basic} (measurability being clear) follows at once by independence of $Y$ and $Z$ \cite[p. 174, Theorem~8.14vi)]{klenke} and the ``taking out what is known'' property (conditional determinism \cite[p. 144, Theorem~1.10a)]{cinlar}) of conditional expectation.
\end{proof}
There is a modification of this proposition, which allows for completions, to wit:

\begin{proposition}[Lemma on conditioning with completions]\label{proposition:lemma_on_conditioning}
Assume now $(\mathcal{F},\PP)$ is complete. Let $Y:(\Omega,\mathcal{F}) \to (S,\mathcal{S})$ and $Z:(\Omega,\mathcal{F})\to (T,\mathcal{T})$ again be two random elements, and $\mathcal{G}$ any sub-$\sigma$-algebra of $\mathcal{F}$, such that $\sigma(Y)\subset \mathcal{G}$ and $\sigma(Z)\perp \mathcal{G}$. Let $f$ be any bounded (or nonnegative, or nonpositive) $\overline{\mathcal{S}\otimes \mathcal{T}}^{\PP_{(Y,Z)}}$/$\mathcal{B}([-\infty,+\infty])$-measurable mapping. Then:
\begin{enumerate}[(i)]
\item\label{loc:i} $(Y,Z)$ is $\mathcal{F}$/$\overline{\mathcal{S}\otimes \mathcal{T}}^{\PP_{(Y,Z)}}$-measurable,
\item\label{loc:ii} $Y$ (respectively $Z$) is $\mathcal{F}$/$\overline{\mathcal{S}}^{\PP_Y}$-measurable (respectively  $\mathcal{F}$/$\overline{\mathcal{T}}^{\PP_Z}$-measurable),
\item\label{loc:iii} $\overline{\PP_Y}$-a.s. in $y\in S$, $f \circ (y,Z)$ is $\mathcal{F}$/$\mathcal{B}([-\infty,+\infty])$-measurable,
\item\label{loc:iv} $(y\mapsto \EE[f\circ (y,Z)])$ is $\overline{\mathcal{S}}^{\PP_Y}$/$\mathcal{B}([-\infty,+\infty])$-measurable  (defining $\EE[f\circ (y,Z)]$ to be, say, $0$, on the $\overline{\PP_Y}$-null set in $y\in S$, on which $f \circ (y,Z)$ is not $\mathcal{F}$/$\mathcal{B}([-\infty,+\infty])$-measurable)
\end{enumerate}
and, $\PP$-a.s., 
\begin{equation}\label{eq:loc}
\EE[f\circ (Y,Z)\vert \mathcal{G}]=(y\mapsto \EE[f\circ (y,Z)])\circ Y.
\end{equation}
\end{proposition}
\begin{proof}
Throughout we use the Image-Measure Theorem \cite[p. 121, Theorem 4.1.11]{dudley}.

First note that $(Y,Z)$ is $\mathcal{F}$/$\mathcal{S}\otimes \mathcal{T}$-measurable, hence $\mathcal{F}$/$\overline{\mathcal{S}\otimes \mathcal{T}}^{\PP_{(Y,Z)}}$-measurable, since $\mathcal{F}$ is $\PP$-complete. Similarly for $Y$ and $Z$. (In both cases apply a generating class argument combining \cite[pp. 101-102, Theorem 3.3.1 and Propositions 3.3.2 \& 3.3.3]{dudley}, cf. also \cite[p. 21, Exercise 8]{kallenberg}.) Thus we have \ref{loc:i} and \ref{loc:ii}. 

Next, the measure spaces $(S,\overline{\mathcal{S}}^{ \PP_Y},\overline{\PP_Y})$ and $(T,\overline{\mathcal{T}}^{ \PP_Z},\overline{\PP_Z})$ are complete and, by \cite[p. 543, Theorem 23.23]{yeh}, $\overline{\overline{\mathcal{S}}^{ \PP_Y}\otimes \overline{\mathcal{T}}^{\PP_Z}}^{\overline{\PP_Y}\times \overline{\PP_Z}}=\overline{\mathcal{S}\otimes \mathcal{T}}^{\PP_{(Y,Z)}}$, since $\PP_Y\times \PP_Z=\PP_{(Y,Z)}$, owing to independence of $Y$ and $Z$. It follows that $f$ is $\overline{\overline{\mathcal{S}}^{ \PP_Y}\otimes \overline{\mathcal{T}}^{\PP_Z}}^{\overline{\PP_Y}\times \overline{\PP_Z}}$/$\mathcal{B}([-\infty,+\infty])$-measurable. The latter allows to conclude \ref{loc:iii} and \ref{loc:iv}, as follows. 

First, by \cite[p. 545, Theorem 23.25(b)]{yeh}, $f(y,\cdot)$ is $\overline{\mathcal{T}}^{\PP_Z}$/$\mathcal{B}([-\infty,+\infty])$-measurable, $\overline{\PP_Y}$-a.s. in $y\in S$. Coupled with \ref{loc:ii}, this yields  \ref{loc:iii}. Second, note that for any $y\in S$ for which $f(y,\cdot)$ is $\overline{\mathcal{T}}^{ \PP_Z}$/$\mathcal{B}([-\infty,+\infty])$-measurable,  $\EE[f\circ (y,Z)]=\int f(y,\cdot)d\overline{\PP_Z}$. Thus \ref{loc:iv} follows by Tonelli's Theorem \cite[p. 546, Theorem 23.26(a)]{yeh}. 

Finally we wish to establish \eqref{eq:loc}. As in Lemma~\ref{eq:loc_basic}, linearity and monotonicity of conditional expectation show that the class of $\overline{\mathcal{S}\otimes \mathcal{T}}^{\PP_{(Y,Z)}}$/$\mathcal{B}([-\infty,+\infty])$-measurable functions $f$ for which \eqref{eq:loc} holds is a monotone class. By the Functional Monotone Class Theorem it will thus be sufficient to consider $f=\mathbbm{1}_\Lambda$ with $\Lambda$ belonging to the $\pi$-system $\{A\times B:(A,B)\in \mathcal{S}\times \mathcal{T}\}\cup \mathcal{N}$, where $\mathcal{N}$ is the set of all $\overline{ \PP_{(Y,Z)}}$-null sets, generating $\overline{\mathcal{S}\otimes \mathcal{T}}^{\PP_{(Y,Z)}}$ \cite[p. 102, Proposition 3.3.2]{dudley}. 

Now, for $\Lambda$ belonging to $\{A\times B:(A,B)\in \mathcal{S}\times \mathcal{T}\}$, \eqref{eq:loc} is the contents of Proposition~\ref{proposition:lemma_on_conditioning_basic}. On the other hand suppose $\Lambda$ is $\overline{\PP_{(Y,Z)}}$-null. Then, $\PP$-a.s., the left-hand side of \eqref{eq:loc} is equal to $0$, since $\overline{\PP_{(Y,Z)}}$ coincides with the law of $(Y,Z)$ on $\overline{\mathcal{S}\otimes \mathcal{T}}^{\PP_{(Y,Z)}}$ and hence $\EE[f\circ (Y,Z)]=\int fd\overline{\PP_{(Y,Z)}}=0$. The right-hand side of \eqref{eq:loc} is nonnegative. To show that it too is $0$, $\PP$-a.s., compute again its expectation using Tonelli's Theorem \cite[p. 546, Theorem 23.26]{yeh} and the fact that by \cite[p. 543, Theorem 23.23]{yeh} $\overline{\PP_{(Y,Z)}}=\overline{\overline{\PP_Y}\times\overline{\PP_z}}$: 
\begin{equation}\label{eq:appendix:non-random:completions}
\int d\overline{\PP_Y}(y)\int d\overline{\PP_Z}(z)f(y,z)=\int d\overline{\overline{\PP_Y}\times \overline{\PP_Z}}f=\int d\overline{\PP_{(Y,Z)}}f=0.
\end{equation}
Thus indeed also the right-hand side of \eqref{eq:loc} equals $0$, $\PP$-a.s., and the proof is complete.
\end{proof}

\end{document}